\RequirePackage{etex}
\RequirePackage{easymat}

\documentclass[12pt]{article}
\usepackage{amsmath,amsthm, amssymb,mathabx}
\usepackage[all]{xy}
\usepackage[active]{srcltx}
\let\absast=\ast
\usepackage{MnSymbol}
\let\ast=\absast
\usepackage{graphicx}
\newcommand{\Aster}{\mathop{\scalebox{1.5}%
{\raisebox{-0.2ex}{$\ast$}}}}

\renewcommand{\le}{\leqslant}
\renewcommand{\ge}{\geqslant}
\newcommand{\ci}{
\begin{picture}(6,6)
\put(3,3.6){\circle*{3}}
\end{picture}}

\newcommand{\T}{\intercal}
\newcommand{\pen}{\!\medstar\!}
\newcommand{\mc}{\mathcal}

\newtheorem{theorem}{Theorem}
\newtheorem{lemma}{Lemma}
\newtheorem{corollary}{Corollary}

\theoremstyle{definition}
\newtheorem{example}{Example}
\newtheorem{definition}{Definition}

\theoremstyle{remark}
\newtheorem{remark}{Remark}

\DeclareMathOperator{\tr}{trace}
\DeclareMathOperator{\diag}{diag}
\begin{document}

\title{Specht's criterion for systems of linear mappings\thanks{Published in  Linear Algebra Appl.  519C (2017) 278--295.}}

\date{}

\author{Vyacheslav Futorny\thanks{Department of Mathematics, University of S\~ao Paulo, Brazil, {futorny@ime.usp.br}}
\and
Roger A. Horn\thanks{Mathematics
Department, University of Utah, Salt Lake City, Utah USA,
{rhorn@math.utah.edu}}
\and
Vladimir V.
Sergeichuk\thanks{Institute of Mathematics,
Kiev, Ukraine, {sergeich@imath.kiev.ua}}
}

\maketitle

\begin{abstract}
W. Specht (1940) proved that two $n\times n$ complex matrices $A$ and $B$ are unitarily similar if and only if $\operatorname{trace} w(A,A^{\ast}) = \operatorname{trace} w(B,B^{\ast})$ for every word $w(x,y)$ in two noncommuting variables. We extend his criterion and its generalizations by N.A. Wiegmann (1961) and N. Jing (2015) to an arbitrary system $\mc A$ consisting of complex or real inner product spaces  and linear mappings among them. We represent such a system by the directed graph $Q(\mc A)$, whose vertices are inner product spaces and  arrows are linear mappings. Denote by $\widetilde Q(\mc A)$ the directed graph obtained by enlarging to $Q(\mc A)$ the adjoint linear mappings. We prove that a system $\mc A$ is transformed by isometries of its spaces to a system $\mc B$ if and only if the traces of all closed directed walks in $\widetilde Q(\mc A)$ and $\widetilde Q(\mc B)$ coincide.


\end{abstract}

\section{Introduction}\label{intr}

Each system of complex inner product spaces and linear mappings among them can be represented by a directed graph, in which the vertices are inner product spaces and the arrows are linear mappings. We reduce the problem of classifying such systems to the problem of classifying complex matrices up to unitary similarity,  apply Specht's criterion for unitary similarity of complex matrices, and obtain a generalization of the following criteria:

\begin{description}
\item[Specht's criterion for unitary similarity] (\cite{spe}; see also \cite[Theorem 2.2.6]{h-j}, \cite{jin}, \cite[Theorem 6.3]{kap}, and \cite{pea}).
Two $n\times n$ complex matrices $A$ and $B$ are unitarily similar if and only if \begin{equation}\label{lur}
\tr w(A,A^{\ast}) = \tr w(B,B^{\ast})
\end{equation}
for every word $w(x,y)$ in two noncommuting variables.

\item[Wiegmann's criterion for simultaneous unitary similarity] (\cite{wie}; see also \cite[Theorem 6.2]{sha}). Let $(A_1,\dots,A_k)$ and $(B_1,\dots,B_k)$ be two $k$-tuples of $n\times n$ complex matrices. There exists a unitary matrix $U$ such that
$(U^{-1}A_1U,\dots,U^{-1}A_kU)=(B_1,\dots,B_k)$ if and only if
\begin{equation}\label{hsr}
\tr w(A_1,A_1^{\ast},\dots,A_k,A_k^{\ast}) = \tr w(B_1,B_1^{\ast},\dots,B_k,B_k^{\ast})
\end{equation}
for every word $w(x_1,y_1,\dots,x_k,y_k)$ in $2k$ noncommuting variables.

\item[Jing's criterion for simultaneous unitary equivalence] (\cite{jin}).
Let $(A_1,\dots,A_k)$ and $(B_1,\dots,B_k)$ be two $k$-tuples of $m\times n$ complex matrices. There exist unitary matrices $U$ and $V$ such that
$(UA_1V,\dots,UA_kV)=(B_1,\dots,B_k)$ if and only if
\begin{equation}\label{vfd}
\begin{split}
&\tr w(A_1^{\ast}A_1,\dots,A_i^{\ast}A_j,\dots,
A_k^{\ast}A_k) \\= &\tr w(B_1^{\ast}B_1,\dots,B_i^{\ast}B_j,\dots,
B_k^{\ast}B_k)
\end{split}
\end{equation}
for every word $w(x_{11},\dots,x_{ij},\dots,x_{kk})$ in $k^2$ noncommuting variables.
\end{description}

Complex $n\times n$
matrices $A$ and $B$ are \emph{unitarily similar} if $B=U^{-1}AU$ for some unitary matrix $U$; they are \emph{complex orthogonally similar} if $B = S^{-1} AS$ for some complex matrix $S$ such that $S^TS =I_n$.
Real $n\times n$ matrices $A$ and $B$ are \emph{real orthogonally similar} if $B = S^{-1} AS$ for some real matrix  $S$ such that $S^TS =I_n$.

Pearcy \cite{pea} (see also \cite[Section 2-6]{kap} and \cite[Section 2.2]{h-j}) noticed that Specht's criterion also holds for real matrices with respect to real orthogonal similarity.
Jing \cite{jin} proved that his, Specht's, and Wiegmann's criteria hold for real matrices with respect to real orthogonal similarity (real orthogonal equivalence in Jing's criterion) and for complex matrices with respect to complex orthogonal similarity (equivalence) if transposed matrices are used instead of conjugate transposed matrices in \eqref{lur}--\eqref{vfd}.

Specht's criterion requires infinitely many tests. Pearcy \cite[Theorem 1]{pea} proved that it suffices to verify the condition \eqref{lur} for all words of length at most $2n^2$. Laffey \cite{laf} showed that it is sufficient to verify \eqref{lur} for all words of length at most the smallest integer that is greater than or equal to $(2n^2+4)/3$, and hence
for all words of length at most $n^2$. A better bound
\begin{equation}\label{vrv}
-2+\frac n2+
n\cdot\sqrt{\frac{2n^2}{n-1}+\frac 14}
\end{equation}
on the sufficient word length
was given by Pappacena \cite{pap}; see also \cite[Theorem 2.2.8]{h-j}.

Two alternative  approaches to the problem of unitary similarity involve different ideas:
\begin{itemize}
  \item
Arveson \cite[Theorems 2 and 3]{arv0} (see also \cite{far2,far,far1}) proved that
if $A,B\in \mathbb C^{n\times n}$ and  $A$ is not unitarily similar to a direct sum of square matrices of smaller sizes, then
$A$ and $B$ are unitarily similar if and only if
\[
\|X\otimes I_n+Y\otimes A\|=\|X\otimes I_n+Y\otimes B\|
\]
for all $X,Y\in \mathbb C^{n\times n}$.
Here, $\|M\|$ is the spectral norm (largest singular value) of $M$.

  \item
Littlewood \cite{lit} (see also \cite{sha}) constructed an algorithm  that reduces each square complex matrix $A$ by unitary similarity transformations to a ``canonical'' matrix $A_{\text{can}}$ in such a way that $A$ and $B$ are unitarily similar if and only if $A_{\text{can}}=B_{\text{can}}$. Littlewood's algorithm was extended in \cite{ser_func,ser_unit} to unitary representations of a quiver.
\end{itemize}

\section{Representations and ${}^{\pen}$unitary representations of a quiver}

\subsection{Representations of a quiver}

Classification problems for systems of linear mappings can be formulated in terms of quivers
and their representations introduced by P.~Gabriel in \cite{gab}. A \emph{quiver} is a directed graph (loops and multiple arrows are allowed); we suppose that its vertices are $1,\dots,t$.
Its \emph{representation} $\mathcal A=(\mathcal A_{\alpha },\mathcal U_v)$ over a field $\mathbb F$ is given by assigning to each vertex $v$ a vector space $\mathcal U_v$ over $\mathbb F$ and to each arrow $\alpha : u \longrightarrow  v$ a linear mapping $\mathcal A_{\alpha}: \mathcal U_u \to \mathcal U_v$.
The vector
\begin{equation}\label{e3y}
\dim \mathcal A:=
(\dim \mathcal U_1,\dots,\dim \mathcal U_t)
\end{equation}
is the \emph{dimension} of the representation $\mathcal A$.

For example, each representation
\begin{equation}\label{kte}
\mathcal A:\ \raisebox{20pt}{\xymatrix{
 &{\mathcal U_1}&\\
 {\mathcal U_2}
  \save !<-2mm,0cm>\ar@(ul,dl)@{->}_{\mathcal A_{\gamma}}\restore
 \ar@{->}[ur]^{\mathcal A_{\alpha}}
\ar@<0.4ex>[rr]^{\mathcal A_{\delta}}
 \ar@<-0.4ex>[rr]_{\mathcal A_{\varepsilon}} &&{\mathcal U_3}
 \ar[ul]_{\mathcal A_{\beta}}
 \save !<2mm,0cm> \ar@(ur,dr)^{\mathcal A_{\zeta}}\restore
 }}
\end{equation}
of the quiver
\begin{equation}\label{2j}
\raisebox{20pt}{\xymatrix{
 &{1}&\\
 {2}\ar@(ul,dl)@{->}_{\gamma}
 \ar@{->}[ur]^{\alpha}
\ar@<0.4ex>[rr]^{\delta}
 \ar@<-0.4ex>[rr]_{\varepsilon} &&{3}
 \ar[ul]_{\beta}
 \ar@(ur,dr)^{\zeta}
 }}
\end{equation}
consists of vector spaces $\mathcal U_1,\
\mathcal U_2,\ \mathcal U_3$ and linear mappings $\mathcal A_{\alpha }$, $\mathcal A_{\beta}$, $\mathcal A_{\gamma }$, $\mathcal A_{\delta }$, $\mathcal A_{\varepsilon }$, $\mathcal A_{\zeta}$.

An \emph{oriented cycle $\pi$ of length $\ell\ge 1$} in a quiver $Q$ is  a sequence of arrows of the form
\begin{equation}\label{iyw}
\pi:\quad\xymatrix{
{v_1}
\ar@{->}@/_2pc/[rrr]^{\alpha_{\ell}}
\ar@{<-}[r]^{\ \alpha_1}&
v_2\ar@{<-}[r]^{\alpha_2\ } &
{\ \cdots\ }&{v_{\ell}}
\ar@{->}[l]_{\quad \alpha_{\ell-1}\ }}
\end{equation}
in which some of the vertices $v_1,v_2,\dots,v_{\ell}$ and some of the arrows $\alpha_1,\dots,\alpha _{\ell}$ may coincide; see \cite[Section 2.1]{rin}. Thus, an oriented cycle is a closed directed walk, in which vertices and arrows may repeat.

For each representation $\cal A$ of $Q$ and any cycle \eqref{iyw},
define the cycle of linear mappings
\[
\mathcal A(\pi):\quad\xymatrix{
{\mathcal U_{v_1}}
\ar@{->}@/_2pc/[rrr]^{\mathcal A_{\alpha_{\ell}}}
\ar@{<-}[r]^{\mathcal A_{\alpha_1}}&
\mathcal U_{v_2}\ar@{<-}[r]^{\mathcal A_{\alpha_2}} &
{\ \cdots\ }&{\mathcal U_{v_{\ell}}}
\ar@{->}[l]_{\ \mathcal A_{\alpha_{\ell-1}}}}
\]
on \eqref{iyw}. Write
\[
\tr \mathcal A(\pi):=\tr\left(\mathcal A_{\alpha_1}
\mathcal A_{\alpha_2}\cdots\mathcal A_{\alpha_{\ell}}\right);
\]
this number \emph{does not depend on the choice of the initial vertex} $v_1$ in the cycle since the trace is invariant under cyclic permutations:
\[
\tr (\mathcal A_{\alpha_1}
\mathcal A_{\alpha_2}\cdots\mathcal A_{\alpha_{\ell}})
=\tr (\mathcal A_{\alpha_{\ell}}\mathcal A_{\alpha_1}\cdots
\mathcal A_{\alpha_{{\ell}-1}})=
\tr (\mathcal A_{\alpha_{{\ell}-1}}\mathcal A_{\alpha_{\ell}}\cdots
\mathcal A_{\alpha_{{\ell}-2}})=\cdots
\]

\subsection{${}^{\pen\!}$Unitary representations of a quiver}

We extend Specht's criterion to systems of linear mappings on complex inner product spaces. A \emph{complex inner product space} (also called a Hermitian space, a finite-dimensional Hilbert space, or a unitary space) is a complex vector space with scalar product given by  a positive definite Hermitian form. We also extend Specht's criterion to systems of linear mappings on \emph{complex Euclidean spaces}, which are complex vector spaces with scalar product given by a  nonsingular symmetric bilinear form.

For convenience in studying
the respective spaces simultaneously, a complex inner product space
is called a \emph{{$^\ast$\!\/}unitary space}, and a complex Euclidean space is called a \emph{$^{\T\!}$unitary space}.

Let $\pen\in\{\T,\Aster\}$.
For each linear mapping $\mathcal A: \mathcal U\to \mathcal V$ between  $^{\pen}$unitary spaces $\mathcal U$ and $\mathcal V$, we define the \emph{adjoint} mapping $\mathcal A^{\pen}:\mathcal V\to \mathcal U$ via
\begin{equation}\label{gkm}
\text{$(\mathcal Ax,y)=(x,\mathcal A^{\pen}y)$\qquad for all $x\in \mathcal U$ and $y\in\mathcal V$.}
\end{equation}

The following definition generalizes the definition of unitary representations of quivers given in \cite{ser_func,ser_unit}.
\begin{definition}\label{hro}
Let $Q$ be a quiver with vertices $1,\dots,t$, and let $\pen\in\{\T,\Aster\}$ be fixed.
\begin{itemize}
  \item A \emph{$^{\pen\!}$unitary representation} $\mathcal A=(\mathcal A_{\alpha },\mathcal U_v)$ of $Q$ is given by assigning to each vertex $v$ a $^{\pen}$unitary space $\mathcal U_v$ and to each arrow $\alpha : u \longrightarrow  v$ a linear mapping $\mathcal A_{\alpha}: \mathcal U_u \to \mathcal U_v$.

  \item Two $^{\pen}$unitary representations $\mathcal A=(\mathcal A_{\alpha },\mathcal U_v)$ and $\mathcal B=(\mathcal B_{\alpha },\mathcal V_v)$ of $Q$ are \emph{isometric} if
there exists a family of isometries (linear bijections that preserve the scalar products)
$\varphi_1:\mathcal U_1\to \mathcal V_1,\dots,
\varphi_t:\mathcal U_t\to \mathcal V_t$ such that the
diagram\\[-4mm]
\begin{equation}\label{mmd}
\begin{split}
\xymatrix{
\mathcal U_u\ar[r]^{\mathcal A_{\alpha} }
\ar[d]_{\varphi_u}
&\mathcal U_v\ar[d]^{\varphi_v}
\\
\mathcal V_u\ar[r]^{\mathcal B_{\alpha} }&
\mathcal V_v}
\end{split}
\end{equation}
is commutative ($\varphi_v\mathcal A_{\alpha }=\mathcal B_{\alpha
}\varphi _u$) for each arrow $\alpha
:u\longrightarrow  v$.
\end{itemize}
\end{definition}

For example, the problem of classifying $^{\pen}$unitary representations of the quiver \eqref{2j}
is the problem of classifying systems
\eqref{kte}
consisting of $^{\pen}$unitary spaces $\mathcal U_1,\
\mathcal U_2,\ \mathcal U_3$ and linear mappings $\mathcal A_{\alpha }$, $\mathcal A_{\beta}$, \dots, $\mathcal A_{\zeta}$.

It is customary to omit the asterisk in the terms ``$^{\Aster\!}$unitary space'' and ``$^{\Aster\!}$unitary representation''.

\section{The main theorem and its corollaries}

\subsection{The main theorem}

For each quiver $Q$, denote by $\widetilde Q$  the quiver with double the number of
arrows obtained  by attaching to $Q$ the arrows $\alpha^{\star} :v\longrightarrow  u$ for all arrows $\alpha :u\longrightarrow  v$ of $Q$.
For example, if $Q$ is the quiver \eqref{2j},
then $\widetilde Q$ is
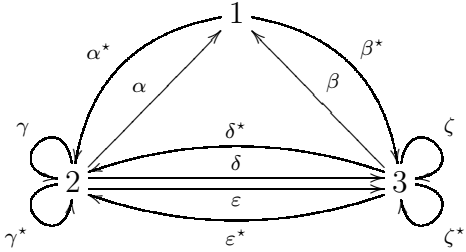
\[{\xymatrix@=50pt{
 &{1}&\\
 {2}
 \save !<0mm,0.5mm>
 \ar@(u,l)_{\gamma}\restore
 \ar@(l,d)_{\gamma^{\star}}
 \ar@<0.4ex>[rr]^{\delta}
 \ar@<-0.4ex>[rr]_{\varepsilon}
 \ar@<0.4ex>@/^1pc/@{<-}[rr]^{\delta^{\star}}
\ar@/^1.5pc/@{<-}[ur]^{\alpha^{\star}}
 \ar@{->}[ur]^{\alpha}
 \ar@<0.4ex>@/^1pc/@{->}
 [rr];[]^{\varepsilon^{\star}} &&{3}
 \ar@/_1.5pc/@{<-}[ul]_{\beta^{\star}}
 \ar@{->}[ul]_{\beta}
 \save !<0mm,0.5mm>
 \ar@(u,r)^{\zeta}\restore
 \ar@(r,d)^{\zeta^{\star}}
 }}
\]

For each ${}^{\pen}$unitary representation $\mathcal A$ of $Q$, we define the ${}^{\pen}$unitary representation $\widetilde {\mathcal A}$ of
$\widetilde Q$ that coincides with $\mathcal A$ on $Q\subset \widetilde Q$ and that assigns to each new arrow $\alpha^{\star} :v\longrightarrow  u$ the mapping $\widetilde{\mathcal A}_{\alpha^{\star}}:=\mathcal A_{\alpha }^{\pen}:\mathcal U_v\to \mathcal U_u$, which is the \emph{adjoint} of $\mathcal A_{\alpha }:\mathcal U_u\to \mathcal U_v$ (see \eqref{gkm}).

The main result of the article is the following theorem, which is proved in Section \ref{huk}.

\begin{theorem}\label{kke}
Let\/ $\pen\in\{\T,\Aster\}$.
\begin{itemize}
  \item[\rm(a)]
Two ${}^{\pen}$unitary representations $\mathcal A$ and $\mathcal B$ of a quiver $Q$ are isometric if and only if
\begin{equation}\label{mcv}
\tr {\widetilde {\mathcal A}}(\pi)=\tr\widetilde{\mathcal B}(\pi)
\end{equation}
for each oriented cycle $\pi$ in the quiver $\widetilde Q$.

  \item[\rm(b)] It suffices to verify \eqref{mcv} for all cycles $\pi$ of length at most
\begin{equation}\label{wwk}
\varphi((r+2)(d_1+\dots+d_t)),
\end{equation}
in which $(d_1,\dots,d_t)$ is the dimension of the representations $\mathcal A$ and $\mathcal B$  $($see \eqref{e3y}$)$, $\varphi (n)$ is any bound
for the sufficient word length in Specht's criterion $($for example, $\varphi (n)$ is $n^2$ or Pappacena's bound \eqref{vrv}$)$, and $r$ is the minimal natural number such that
\begin{equation}\label{fvk}
\frac{r(r+1)}2\ge\max\{m_{ij}\,|\, \text{$i$ and $j$ are vertices of $Q$}\},
\end{equation}
in which $m_{ij}$ is the number of arrows from $j$ to $i$ in $Q$.
\end{itemize}
\end{theorem}

\subsection{The main theorem in matrix form}

We say that a square complex matrix $A$ is \emph{${}^{\pen}$unitary} ($\pen\in\{\T,\Aster\}$) if $A^{\pen}A=I$. Thus, ${}^{\T}$unitary matrices are complex orthogonal matrices and
$^{\ast}$unitary matrices are unitary matrices.

A basis of a ${}^{\pen}$unitary space $\mathcal U$ is \emph{orthonormal} if the scalar product in this basis is given by the identity matrix. The change of basis matrix from an orthonormal basis to an orthonormal basis is a $^{\pen}$unitary matrix. If $[x]$ is the coordinate vector of $x\in \mathcal U$ in an orthonormal basis, then $(x,y)=[x]^{\pen\,}[y]$ for all $x,y\in \mathcal U$. If $\mathcal A:\mc U\to\mc V$ is a linear mapping between complex inner ${}^{\pen}$product spaces  and $A$ is its matrix in some orthonormal bases of  $\mathcal U$ and $\mc V$, then $ A^{\pen}$ is the matrix of the adjoint  mapping  $\mathcal A^{\pen}:\mc V\to\mc U$ (see \eqref{gkm}).

Each ${}^{\pen}$unitary representation $\mathcal A$ in \eqref{kte} can be given by
the sequence $A=(A_{\alpha
},\, A_{\beta},\, \dots,\, A_{\zeta})$ of
matrices of the linear mappings $\mathcal
A_{\alpha },\, \mathcal A_{\beta},\, \dots,\,
\mathcal A_{\zeta}$ in some orthonormal bases
of the spaces $\mathcal U_1,\
\mathcal U_2,\ \mathcal U_3$. The representation $\cal A$ in other orthonormal bases is given by the sequence
\begin{equation}\label{grp}
(U_1^{-1}A_{\alpha }U_2,\,
U_1^{-1}A_{\beta}U_3,\,
U_2^{-1}A_{\gamma  }U_2,\,
U_3^{-1}A_{\delta  }U_2,\,
U_3^{-1}A_{\varepsilon }U_2,\,
U_3^{-1}A_{\zeta}U_3)
\end{equation}
in which $U_1,\,U_2,\,U_3$ are the change of
basis matrices; they are arbitrary  ${}^{\pen}$unitary
matrices of suitable sizes.
Thus, the problem of classifying  ${}^{\pen}$unitary
representations of the quiver \eqref{2j}
reduces to the problem of classifying matrix
sequences  $(A_{\alpha },\, \dots,\, A_{\zeta})$
up to transformations of the form \eqref{grp}.
This example leads to the following definition.
\begin{definition}\label{hrs}
Let $\pen\in\{\T,\Aster\}$ and let $Q$ be a quiver with vertices $1,\dots,t$.
\begin{itemize}
  \item A \emph{complex matrix
      representation $A$ of
      dimension $(d_1,\dots,d_t)$}
      of $Q$ is given by assigning
      to each
      arrow
      $\alpha:u\longrightarrow v$
      a complex matrix ${A}_{\alpha}$ of size $d_v\times d_u$ (we take $d_i:=0$ if the vertex $i$ does not have arrows).

  \item Two complex matrix representations
      $A$ and $B$ of $Q$ are \emph{${}^{\pen\!}$unitarily isometric} if there
      exist ${}^{\pen}$unitary matrices
      $U_1,\dots, U_t$ such that
\begin{equation}\label{ljt}
B_{\alpha}=
    U^{-1}_v{A}_{\alpha}{U}_u\qquad
\hbox{for every arrow $\alpha:
u\longrightarrow v$}.
\end{equation}
\end{itemize}
\end{definition}

For example, two complex matrix representations
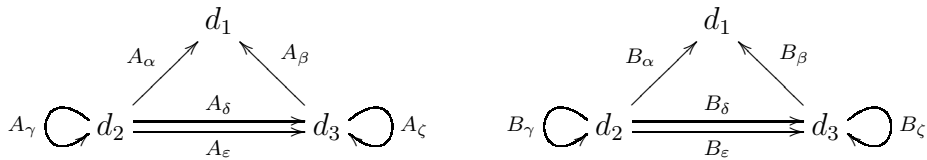
\[
\xymatrix{
 &{d_1}&\\
 {d_2}
  \save !<-2mm,0cm>\ar@(ul,dl)@{->}_{ A_{\gamma}}\restore
 \ar@{->}[ur]^{A_{\alpha}}
\ar@<0.4ex>[rr]^{A_{\delta}}
 \ar@<-0.4ex>[rr]_{A_{\varepsilon}} &&{d_3}
 \ar[ul]_{A_{\beta}}
 \save !<2mm,0cm> \ar@(ur,dr)^{ A_{\zeta}}\restore
 }\qquad
\xymatrix{
 &{d_1}&\\
 {d_2}
  \save !<-2mm,0cm>\ar@(ul,dl)@{->}_{ B_{\gamma}}\restore
 \ar@{->}[ur]^{B_{\alpha}}
\ar@<0.4ex>[rr]^{B_{\delta}}
 \ar@<-0.4ex>[rr]_{B_{\varepsilon}} &&{d_3}
 \ar[ul]_{B_{\beta}}
 \save !<2mm,0cm> \ar@(ur,dr)^{ B_{\zeta}}\restore
 }
\]
of the quiver \eqref{2j} are $^{\pen}$unitarily isometric if and only if $B=(B_{\alpha },\dots,B_{\zeta})
$ is of the form \eqref{grp}.

The principle formalized in the following obvious lemma reduces the problem of classifying ${}^{\pen}$unitary representations up to isometry to the problem of classifying complex matrix representations up to $^{\pen}$unitary isometry.

\begin{lemma}\label{vfo}
Let $\cal A$ and $\cal B$ be two ${}^{\pen}$unitary representations of a quiver.
Choosing orthonormal bases in their spaces, we get two complex matrix representations $A$ and $B$. Then
$\cal A$ and $\cal B$ are isometric if and only if $A$ and $B$ are $^{\pen}$unitarily isometric.
\end{lemma}

For each oriented cycle \eqref{iyw} in a quiver $Q$ and each complex matrix representation $A$
of $Q$, we write $ A(\pi):=A_{\alpha_1} A_{\alpha_2}\cdots A_{\alpha_{\ell}}.$
The following theorem is equivalent to Theorem \ref{kke} due to Lemma \ref{vfo}.

\begin{theorem}\label{kju}
Let\/ $\pen\in\{\T,\Aster\}$.
\begin{itemize}
  \item[\rm(a)]
Two complex matrix representations $A$ and $B$ of a quiver $Q$ are $^{\pen}$unitarily isometric if and only if
\begin{equation}\label{q1i}
\tr {\widetilde A}(\pi)=\tr{\widetilde B}(\pi)
\end{equation}
for each oriented cycle $\pi$ in the quiver $\widetilde Q$.

\item[\rm(b)]
It suffices to verify \eqref{q1i} for all cycles $\pi$ of length at most \eqref{wwk}.
\end{itemize}
\end{theorem}

\subsection{Corollaries}

A \emph{Euclidean representation} $\mathcal A=(\mathcal A_{\alpha },\mathcal U_v)$ of a quiver $Q$ is defined in \cite{ser_unit} as a list of Euclidean spaces $\mathcal U_v$ assigned to all vertices $v$ and linear mappings $\mathcal A_{\alpha}: \mathcal U_u \to \mathcal U_v$ assigned to all arrows $\alpha : u \longrightarrow  v$. Two Euclidean representations $\mathcal A=(\mathcal A_{\alpha },\mathcal U_v)$ and $\mathcal B=(\mathcal B_{\alpha },\mathcal V_v)$ of $Q$ are \emph{isometric} if
there exists a family of isometries
$\varphi_1:\mathcal U_1\to \mathcal V_1,\dots,
\varphi_t:\mathcal U_t\to \mathcal V_t$ such that the diagram
\eqref{mmd}
is commutative for each arrow $\alpha
:u\longrightarrow  v$.

We say that a matrix representation $A$ of $Q$
is \emph{real} if all its matrices are real.
Two real matrix representations $A$ and $B$ are
\emph{real orthogonally isometric} if there exist real
orthogonal matrices $U_1,\dots, U_t$ such
that \eqref{ljt} holds.

\begin{corollary}\label{jmo}
\begin{itemize}
  \item[\rm(a)]
Two Euclidian representations $\mathcal A$ and $\mathcal B$ of a quiver $Q$ are isometric if and only if $\tr {\widetilde{\mathcal A}}(\pi)=\tr{\widetilde{\mathcal B}}(\pi)$ for each oriented cycle $\pi$ in the quiver $\widetilde Q$.

  \item[\rm(b)]
Two real matrix representations $A$ and $B$ of a quiver $Q$ are isometric if and only if $\tr {\widetilde A}(\pi)=\tr{\widetilde B}(\pi)$ for each oriented cycle $\pi$ in the quiver $\widetilde Q$.

  \item[\rm(c)] It suffices to verify the equalities in {\rm(a)} and {\rm(b)} for all cycles $\pi$ of length at most \eqref{wwk}.

\end{itemize}
\end{corollary}

\begin{proof}
Let $\cal A$ and $\cal B$ be two Euclidean
representations of a quiver.
Choosing orthonormal bases in their spaces, we
get two real matrix representations $A$ and
$B$. Then $\cal A$ and $\cal B$ are
isometric if and only if $A$ and $B$ are
orthogonally isometric, and so (a) follows from (b). The statement (b) follows from Theorem \ref{kju} due to
the following statement proved in \cite[Theorem
4.1(a)]{ser_unit}:
\begin{quote}
two real matrix representations of a quiver are real
orthogonally isometric if and only if they are
unitarily isometric.
\end{quote}
(In particular, two lists of real matrices are simultaneously real
orthogonally similar if and only if they are simultaneously
unitarily similar; see \cite[Theorem
65]{kap} or \cite[Theorem 2.5.21]{h-j}.)
\end{proof}

\begin{corollary}\label{jrf}
Applying Theorem \ref{kju} and Corollary \ref{jmo} to complex and real matrix  representations of the quivers

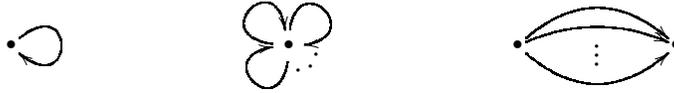
\[
\xymatrix{
*{\ci}
\ar@(ur,dr)}
\qquad     \qquad \qquad      \qquad
\xymatrix{
{\ci}
\ar@(u,r)
\save !<-1pt,0.5pt>
 \ar@(l,u)\restore
 \ar@(d,l)
 \save[]+<2ex,-0.8ex>*
{\cdot}
\restore
\save[]+<1.6ex,-1.7ex>*
{\cdot}
\restore
\save[]+<0.7ex,-2ex>*
{\cdot}
\restore
 }
     \qquad \qquad      \qquad
\xymatrix{
*{\ci}
  \ar@/^1.2pc/[rr]
\ar@/^0.6pc/[rr]
  \ar@/_1.2pc/[rr]
  &\save[]+<0ex,-0.3pc>*
{\vdots}
\restore&
*{\ci}}
\]
we get Specht's, Wiegmann's, and Jing's criteria {\rm(see the beginning of Section \ref{intr}).}
\end{corollary}

Applying Theorem \ref{kju} and Corollary \ref{jmo} to complex and real matrix  representations of the quiver\\[-3ex]
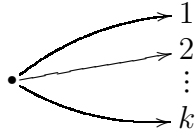
\[
\xymatrix@R=-3pt@C=2cm
{
&1\\&\\&2\\
*{\ci}\ar@/^/@{->}[ruuu]
\ar@{->}[ru]
\ar@/_/@{->}[rdd]
&\vdots
\\&\\&k\\
}
\]
with $k\ge 1$, we obtain the following criterion.

\begin{corollary}\label{jpt}
\begin{itemize}
  \item[\rm(a)]  Let $A_1,\dots,A_k$ and $B_1,\dots,B_k$ $(k\ge 1)$
be complex matrices with $m$ rows. Suppose that $A_i$ and $B_i$ have $n_i$ columns, $i=1,2,\dots,k$. Then there exist unitary matrices $U,V_1,\dots,V_k$ such that
\[
(UA_1V_1,\dots,UA_kV_k)=(B_1,\dots,B_k)
\]
if and only if
\begin{equation}\label{vv1}
\tr w(A_1^{\ast}A_1,\dots,A_k^{\ast}A_k) = \tr w(B_1^{\ast}B_1,\dots,B_k^{\ast}B_k)
\end{equation}
for every word $w(x_1,\dots,x_k)$ in noncommuting variables.

  \item[\rm(b)]
Let $A_1,\dots,A_k$ and $B_1,\dots,B_k$ $(k\ge 1)$
be real (respectively, complex) matrices with $m$ rows.  Suppose that $A_i$ and $B_i$ have $n_i$ columns, $i=1,2,\dots,k$. Then there exist real orthogonal (respectively, complex orthogonal)  matrices $U,V_1,\dots,V_k$ such that
\[
(UA_1V_1,\dots,UA_kV_k)=(B_1,\dots,B_k)
\]
if and only if
\begin{equation}\label{vv2}
\tr w(A_1^{\T}A_1,\dots,A_k^{\T}A_k) = \tr w(B_1^{\T}B_1,\dots,B_k^{\T}B_k)
\end{equation}
for every word $w(x_1,\dots,x_k)$ in noncommuting variables.

  \item[\rm(c)] It suffices to verify \eqref{vv1} and \eqref{vv2} for all words of length at most $\varphi (3(m+n_1+\dots+n_k))$, in which $\varphi (n)$ is any bound
for the sufficient word length in Specht's criterion $($see Theorem {\rm\ref{kke}(b))}.

\end{itemize}
\end{corollary}

Denote by $Q_t$ the \emph{complete quiver} with vertices $1, 2, \dots,t$; that is, the quiver in which each vertex has exactly one loop and
every pair of distinct vertices is connected by a pair of arrows (one in each direction). For example,

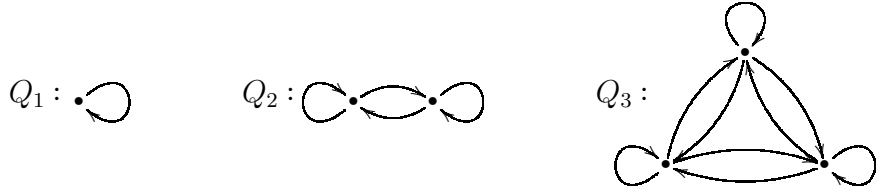
\[
\raisebox{-3ex}
{$Q_1:$ \xymatrix{
*{\ci}
\ar@(ur,dr)}
         \quad\qquad\qquad
$Q_2:\quad\ $ \xymatrix{
*{\ci}
\ar@(dl,ul)
\ar@/_/@{<-}[r]\ar@/^/[r]
&*{\ci}\ar@(ur,dr)}\quad\qquad\qquad $Q_3:$
}
\xymatrix@R=35pt{
&*{\ci}
\ar@(ul,ur)
&\\
*{\ci}
\ar@(dl,ul)
\ar@/_/@{<-}[ru]\ar@/^/[ru]
\ar@/_/@{<-}[rr]\ar@/^/[rr]
&&*{\ci}\ar@(ur,dr)
\ar@/_/@{<-}[lu]\ar@/^/[lu]}\\[3mm]
\]
Applying Theorem \ref{kke} to representations of $Q_t$, we get the following criterion.

\begin{corollary}\label{kit}
Let
\[
A=\begin{bmatrix}
    A_{11}&\cdots & A_{1t} \\
    \vdots&\ddots&\vdots\\
    A_{t1}&\cdots & A_{tt} \\
  \end{bmatrix},\qquad
B=\begin{bmatrix}
   B_{11}&\cdots & B_{1t} \\
    \vdots&\ddots&\vdots\\
    B_{t1}&\cdots & B_{tt} \\
  \end{bmatrix}
\]
be conformally partitioned $m\times m$ complex matrices, in which all diagonal blocks are square. Let\/ $\pen\in\{\T,\Aster\}$ be fixed.
The following statements are equivalent:
\begin{itemize}
  \item[\rm(i)]
   $U^{-1}AU=B$, in which $U=U_1\oplus\dots\oplus U_t$ and each $U_i$ is $^{\pen}$unitary and the same size as $A_{ii}$.

  \item[\rm(ii)] The equality
\begin{equation}\label{gula}
\tr\bigl(A_{i_1i_2}^{(\varepsilon_1)}
A_{i_2i_3}^{(\varepsilon_2)}\cdots
A_{i_{\ell-1}i_{\ell}}^{(\varepsilon_{\ell-1})}
A_{i_\ell i_1}^{(\varepsilon_{\ell})}\bigr)=
\tr\bigl(B_{i_1i_2}^{(\varepsilon_1)}
B_{i_2i_3}^{(\varepsilon_2)}\cdots
B_{i_{\ell-1}i_{\ell}}^{(\varepsilon_{\ell-1})}
B_{i_{\ell}i_1}^{(\varepsilon_{\ell})}\bigr)
\end{equation}
with
\begin{equation}\label{kmj}
A_{ij}^{(\varepsilon)}:=
  \begin{cases}
    A_{ij} & \hbox{if }\varepsilon =1, \\
    A_{ji}^{\varepsilon} & \hbox{if }\varepsilon \ne 1,
  \end{cases}
                   \quad\quad B_{ij}^{(\varepsilon)}:=
  \begin{cases}
    B_{ij} & \hbox{if }\varepsilon =1, \\
    B_{ji}^{\varepsilon} & \hbox{if }\varepsilon \ne 1
  \end{cases}
\end{equation}
holds for all $\varepsilon_1,\dots,
\varepsilon_{\ell}\in\{1,\pen\}$, all $i_1,\dots,
i_{\ell}\in\{1,\dots,t\}$, and every natural number  $\ell$.

\item[\rm(iii)] The equality \eqref{gula} holds for all $\varepsilon_1,\dots,
\varepsilon_{\ell}\in\{1,\pen\}$, all $i_1,\dots,
i_{\ell}\in\{1,\dots,t\}$, and all  $\ell\le \varphi (3m)$, in which $m\times m$ is the size of $A$ and $B$, and $\varphi (n)$ is any bound
for the sufficient word length in Specht's criterion $($see Theorem {\rm\ref{kke}(b))}.

\end{itemize}
\end{corollary}

\section{Proof of Theorems \ref{kke} and \ref{kju}}\label{huk}

It suffices to prove Theorem \ref{kju} since it is equivalent to Theorem \ref{kke}.

To prove Theorem \ref{kju}, we reduce the problem of classifying complex matrix representations of a quiver up to $^{\pen}$unitary isometry to the problem of classifying complex matrices up to $^{\pen}$unitary similarity. We then apply Specht's criterion for matrices under unitary similarity and its generalization by Jing \cite{jin} to matrices under complex orthogonal similarity.

\subsection{From matrix representations of a quiver up to $^{\pen\!}$unitary isometry to matrices up to $^{\pen\!}$unitary similarity}
\label{bdr}

Let $\pen\in\{\T,\Aster\}$.
For each quiver $Q$ and its complex matrix representation $A$, we construct a square complex matrix $M_Q(A)$ such that
\begin{equation}\label{jdfb}
\parbox[c]{0.8\textwidth}{matrix representations $A$ and $B$ are $^{\pen}$unitarily isometric\\
$\Longleftrightarrow$ $M_Q(A)$ and $M_Q(B)$ are $^{\pen}$unitarily similar.}
\end{equation}
Examples of $M_Q(A)$ are given in \cite{ser_func} and \cite[Section 2.3]{ser_unit}, in which Littlewood's algorithm for reducing  complex matrices to canonical form under unitary similarity is extended to unitary representations of quivers.
An analogous construction was used in \cite[Lemma 2]{ger} to reduce the problem of classifying $(p+q)$-tuples of complex $n\times n$ matrices $(A_1,\dots,A_p;B_1,\dots,B_q)$ up to transformations \[(U^{-1}A_1U,\dots,U^{-1}A_pU;\:U^{\T}B_1U,\dots,
U^{\T}B_qU),\qquad U\text{ is unitary}\] to the problem of classifying square
matrices up to unitary similarity.

\subsubsection{An example}
\label{mtp}
Let $A=(A_{\alpha },\dots,A_{\zeta})$ be a complex matrix representation of the quiver \eqref{kte}. Define the matrix
\begin{equation}\label{mmf}
M(A):=\left[\begin{array}{ccc|cc|cc}
I&0&0&0&0&A_{\alpha }&A_{\beta }\\
0&2I&0&I&0&A_{\gamma}&0\\
0&0&3I&0&I&A_{\delta }&A_{\zeta}\\ \hline
0&0&0&4I&0&I&0\\
0&0&0&0&5I&A_{\varepsilon }&I\\ \hline
0&0&0&0&0&6I&0\\
0&0&0&0&0&0&7I
\end{array}\right].
\end{equation}
For each complex matrix representation $B=(B_{\alpha },\dots,B_{\zeta})$ of $Q$ of the same dimension as $A$, we replace the blocks $A_{\alpha },\dots,A_{\zeta}$ of $M(A)$ with $B_{\alpha },\dots,B_{\zeta}$ and denote the matrix obtained by $M(B)$.
Let us prove that
\eqref{jdfb} holds with
$M(A)$ and $M(B)$ instead of
$M_Q(A)$ and $M_Q(B)$.

$\Longrightarrow$. Let $A$ and $B$ be $^{\pen}$unitarily isometric. Then $B$ is represented in the form \eqref{grp}, in which $U_1,U_2,U_3$ are $^{\pen}$unitary matrices. Writing
\begin{equation}\label{n6e}
U=\diag(U_1,U_2,U_3,U_2,U_3,U_2,U_3),
\end{equation}
we obtain
\begin{equation}\label{feo}
U^{-1}M(A)U=\quad\begin{MAT}(@){rccccccc}
&\scriptstyle U_1&\scriptstyle U_2&\scriptstyle U_3&\scriptstyle U_2&\scriptstyle U_3&\scriptstyle U_2&\scriptstyle U_3\\
\scriptstyle U_1^{-1}&I&&&&&A_{\alpha }&A_{\beta }\\
\scriptstyle U_2^{-1}&&2I&&I&&A_{\gamma}&\\
\scriptstyle U_3^{-1}&&&3I&&I&A_{\delta }&A_{\zeta}\\
\scriptstyle U_2^{-1}&&&&4I&&I&\\
\scriptstyle U_3^{-1}&&&&&5I&A_{\varepsilon }&I\\
\scriptstyle U_2^{-1}&&&&&&6I&\\
\scriptstyle U_3^{-1}&&&&&&&7I
\addpath{(1,0,4)rrrrrrruuuuuuulllllllddddddd}
\addpath{(1,4,2)rrrrrrr}
\addpath{(1,2,2)rrrrrrr}
\addpath{(6,0,2)uuuuuuu}
\addpath{(4,0,2)uuuuuuu}\\
\end{MAT}\quad=M(B).
\end{equation}

$\Longleftarrow$.
Let $M(A)$ and $M(B)$ be $^{\pen}$unitarily similar; that is,  $M(A)U=UM(B)$ with $^{\pen}$unitary $U$. Partition $U=[U_{ij}]$ conformally to $M(A)$. Equating the blocks of $M(A)U$ and $UM(B)$ along the block diagonals starting from the lower left corner, we find that $U$ is upper block triangular. Since $U$ is $^{\pen}$unitary, it is block diagonal. Equating the blocks of $M(A)U$ and $UM(B)$ at the places of $I$'s in \eqref{mmf}, we find that $U$ has the form \eqref{n6e}. By \eqref{feo}, $A_{\alpha },A_{\beta },\dots,A_{\zeta}$ are transformed as in \eqref{grp}, which proves \eqref{jdfb}.

\subsubsection{The general case}

\begin{definition}\label{bit}
Let $Q$ be a quiver with vertices $1,\dots,t$. For each pair of vertices $(i,j)$, let
\begin{equation}\label{,,g}
\alpha _{ij:1},\ \alpha _{ij:2}, \ \dots,\ \alpha_{ij:m_{ij}}:j\longrightarrow i
\end{equation}
be all the arrows from $j$ to $i$ (the number $m_{ij}$ is called the \emph{multiplicity of} $j\longrightarrow i$). Let $r$ be the minimal natural number such that $r(r+1)/2\ge\max{m_{ij}}$ (see \eqref{fvk}).
Define the $(r+2)t\times (r+2)t$ partitioned matrix \begin{equation}\label{rwt}
M_Q(x):= \begin{bmatrix}
D_1 & I_t& X_1&X_2&X_3&\dots&X_r
 \\
&D_2&I_t&X_{r+1}&X_{r+2}&\dots&X_{2r-1}
  \\
&&D_3&I_t&X_{2r}&\dots&X_{3r-3}\\
&&& \ddots&\ddots&\ddots&\vdots\\
&&&&D_r&I_t&X_{\frac{r(r+1)}2}\\
&&&&&D_{r+1}&I_t\\
0&&&&&&D_{r+2}\\
 \end{bmatrix},
\end{equation}
in which all blocks are $t\times t$,
\[
D_1:=\diag(1,2,\dots,t),\
D_2:=\diag(t+1,t+2,\dots,2t),\ \dots,
\]
(thus, the main diagonal of $M_Q(x)$ is $(1,2,3,\dots,(r+2)t)$)
and
\[
X_{\xi}:=
\begin{bmatrix}
  x_{11:\xi} &\dots&
  x_{1t:\xi} \\
  \vdots&\ddots&\vdots\\
 x _{t1:\xi} &\dots&
 x_{tt:\xi} \\
\end{bmatrix},\quad
x_{ij:\xi}:=
     0\hbox{ if }\xi>m_{ij}.
\]
\end{definition}

Thus, $M_Q(x)$ depends on parameters
\begin{equation}\label{b97}
x _{ij:1},\ x_{ij:2},\ \dots,\ x_{ij:m_{ij}}\qquad (i,j=1,\dots,t),
\end{equation}
which correspond to the arrows \eqref{,,g}.

\begin{example}
If $Q$ is the quiver \eqref{2j}, then
\begin{equation}\label{nkl}
M_Q(x)=\left[
\begin{array}{c|c|c|c}
\vphantom{\text{\LARGE Q}}
 \!\!\begin{smallmatrix}
   {\scriptstyle 1} \\&
   {\scriptstyle 2}\\&&
   {\scriptstyle 3}
  \end{smallmatrix}\!\!\!&
 \!\!\begin{smallmatrix}
   {\scriptstyle 1} \\&
   {\scriptstyle 1}\\&&
   {\scriptstyle 1}
  \end{smallmatrix}\!\!\!
  &X_1&X_2
                           \\[2mm]\hline
  0\vphantom{\text{\LARGE Q}}&\!\!
   \begin{smallmatrix}
   {\scriptstyle 4} \\&
   {\scriptstyle 5}\\&&
   {\scriptstyle 6}
  \end{smallmatrix}\!\!\!&
 \!\!\begin{smallmatrix}
   {\scriptstyle 1} \\&
   {\scriptstyle 1}\\&&
   {\scriptstyle 1}
  \end{smallmatrix}\!\!\!
  & 0
                                 \\[2mm]\hline
  0\vphantom{\text{\LARGE Q}} &0&\!\!\begin{smallmatrix}
   {\scriptstyle 7} \\&
   {\scriptstyle 8}\\&&
   {\scriptstyle 9}
  \end{smallmatrix}\!\!\! &
  \!\!\begin{smallmatrix}
   {\scriptstyle 1} \\&
   {\scriptstyle 1}\\&&
   {\scriptstyle 1}
  \end{smallmatrix}\!\!\!
                               \\[2mm]\hline
  0\vphantom{\text{\LARGE Q}}&0&0&\!\!\!\begin{smallmatrix}
   {\scriptstyle 10}\! \\&
   \!{\scriptstyle 11}\!\\&&
   \!{\scriptstyle 12}\!\!
  \end{smallmatrix}\!\!\!
  \\[2mm]
\end{array}\right],
\end{equation}
in which
\[
X_1:=\begin{bmatrix}
0&x_{\alpha }&x_{\beta }\\
0&x_{\gamma }&0\\
0&x_{\delta}&x_{\zeta}
     \end{bmatrix},\qquad
X_2:=\begin{bmatrix}
0&0&0\\
0&0&0\\
0&x_{\varepsilon}&0
     \end{bmatrix}.
\]
\end{example}

\begin{lemma}\label{njp}
Let $Q$ be a quiver with vertices $1,\dots,t$ and arrows \eqref{,,g}, and let $M_Q(x)$ be the parameter matrix \eqref{rwt}.
For each complex matrix representation $A$ of $Q$, denote by $M_Q(A)$ the block matrix obtained from $M_Q(x)$ by replacing the parameters \eqref{b97} with
\[
A_{\alpha _{ij:1}},\ A_{\alpha _{ij:2}}, \ \dots,\ A_{\alpha_{ij:m_{ij}}},
\]
replacing each other nonzero entry $a$ with the scalar block $aI$, and replacing the zero entries with the zero blocks of suitable sizes. Then $M_Q(A)$ is  correctly constructed and \eqref{jdfb} holds.
\end{lemma}

\subsubsection{Proof of Lemma \ref{njp}}\label{dru}
Let $A$ be a complex matrix representation of dimension $(d_1,\dots,d_t)$ of $Q$. Substituting it into \eqref{rwt}, we get the block matrix
\begin{equation}\label{rkr}
 M_Q(A)=\begin{bmatrix}
  \widehat D_1 & \widehat I_d&
  \widehat{A}_1&\dots&\widehat{A}_{r}\\
  &\widehat D_2 &
  \widehat I_d&\ddots&\vdots\\
 &&\widehat D_{3}&\ddots&\widehat{A}_{\frac{r(r+1)}2}\\
 &&&\ddots&\widehat I_d\\
 0&&&&\widehat D_{r+2}\\
 \end{bmatrix},
\end{equation}
in which $d:=d_1+\dots+d_t$, each $t\times t$ block of \eqref{rwt} becomes a $d\times d$ block that is partitioned into $t$ horizontal strips of sizes $d_1,\dots,d_t$ and $t$ vertical strips of the same sizes; namely,
\begin{align*}
&\widehat I_d=\diag(I_{d_1},I_{d_2},\dots,I_{d_t})\\
&\widehat D_1=\diag(1I_{d_1},2I_{d_2},\dots,tI_{d_t})\\
&\widehat D_2=\diag((t+1)I_{d_1},(t+2)I_{d_2},\dots,2tI_{d_t}) \\ & \ \vdots
\end{align*}
 and
 \begin{equation}\label{lsk}
\widehat{A}_{\xi}:=
\begin{bmatrix}
  A_{11:\xi} &\dots&A_{1t:\xi} \\
  \vdots&\ddots&\vdots\\
   A_{t1:\xi} &\dots&A_{tt:\xi} \\
\end{bmatrix} \quad\text{with }
 A_{ij:\xi}:=
   \begin{cases}
     A_{\alpha _{ij:\xi}} & \hbox{if }\xi\le
     m_{ij}, \\
     0 & \hbox{if }\xi>m_{ij}.
   \end{cases}
 \end{equation}

Let us prove that \eqref{jdfb} holds for every two complex matrix representations $A$ and $B$ of $Q$.

$\Longrightarrow$. Let $A$ and $B$ be $^{\pen}$unitarily isometric; that is, \eqref{ljt} holds for some $^{\pen}$unitary matrices $U_1,\dots,U_t$. Then $M_Q(B)=U^{-1}M_Q(A)U$ with \[U:=\diag(U_1,\dots,U_t;\dots;U_1,\dots,U_t).
\]

$\Longleftarrow$.
Suppose that $M_Q(A)U=UM_Q(B)$ with a $^{\pen}$unitary matrix $U$.
Partition $U$ into $(r+2)^2$ blocks $U_{ij}$ conformally to \eqref{rkr}.
Equating the blocks of $M_Q(A)U$ and $UM_Q(B)$ along the block diagonals starting from the lower left corner, we find that $U$ is upper block triangular. Since $U$ is $^{\pen}$unitary, $U$ is block diagonal; that is, it has the form $U=\diag(U^{(1)},\dots, U^{(r+2)})$. Since $\widehat D_iU^{(i)}=U^{(i)}\widehat D_i$, each $U^{(i)}$ is block diagonal too: $U^{(i)}=\diag(U^{(i)}_1,\dots, U^{(i)}_t).$ Equating the blocks of $M_Q(A)U$ and $UM_Q(B)$ at the places of $\widehat I_d$, we find that $U^{(1)}=\dots=U^{(r+2)}$, and so
\[
U=\diag(U_1,\dots, U_t;\,U_1,\dots, U_t;\,\dots\, ; U_1,\dots, U_t),
\]
in which every $U_i$ is a $d_i\times d_i$ $^{\pen}$unitary matrix. The equalities
\[
\widehat{A}_{\xi}\diag(U_1,\dots, U_t)
=\diag(U_1,\dots, U_t)\widehat{B}_{\xi},\quad \xi=1,\dots,{r(r+1)}/2
\]
ensure that $A_{\alpha _{ij:\xi}}U_j= U_iB_{\alpha _{ij:\xi}}$ for each arrow $\alpha _{ij:\xi}:j\longrightarrow i$ of $Q$, and so the complex matrix representations $A$ and $B$ are $^{\pen}$unitarily isometric. The proof of Lemma \ref{njp} is complete.

\begin{remark} \label{rrem}
A matrix that is simpler than \eqref{rwt} can be constructed for most concrete quivers. For example,
Section \ref{mtp} shows that the matrix
\begin{equation*}\label{mjr}
M(x):=\begin{bmatrix}
1&0&0&0&0&x_{\alpha }&x_{\beta }\\
0&2&0&1&0&x_{\gamma}&0\\
0&0&3&0&1&x_{\delta }&x_{\zeta}\\
0&0&0&4&0&1&0\\
0&0&0&0&5&x_{\varepsilon}&1\\
0&0&0&0&0&6&0\\
0&0&0&0&0&0&7
\end{bmatrix}
\end{equation*}
can be used in Lemma \ref{njp} instead of \eqref{nkl}.
\end{remark}

\begin{remark}
\label{jyr}
A quiver is \emph{unitarily wild} if the problem of classifying its unitary representations \emph{contains} the problem of classifying unitary representations of the quiver ${\ci\!\!\righttoleftarrow}$; that is, it contains the problem of classifying square complex matrices up to unitary similarity.
By Lemma \ref{njp}, the problem of classifying unitary representations of each quiver is \emph{contained} in the problem of classifying  unitary representations of the quiver ${\ci\!\!\righttoleftarrow}$.
Therefore, the problems of classifying unitary representations have
the same complexity for all unitarily wild quivers. Moreover, a classification of unitary representations of any of them would imply the classification of unitary representations of each quiver.
By \cite[Section 2.3]{ser_unit}, all connected quivers are unitarily wild, except for the simplest quivers $\ci$ and ${\ci\!\!\longrightarrow\!\!\ci}$.
The notion of unitarily wild matrix problems is analogous to the notion of wild matrix problems: a matrix problem is \emph{wild} if it contains the problem of classifying matrix pairs up to similarity. By \cite{bel}, the latter problem contains the problem of classifying representations of an arbitrary quiver and an arbitrary partially ordered set.
\end{remark}

\subsection{Proof of Theorem \ref{kke}}\label{key}

(a) Let $\mathcal A=(\mathcal A_{\alpha },\mathcal U_v)$ and $\mathcal B=(\mathcal B_{\alpha },\mathcal V_v)$ be two ${}^{\pen}$unitary representations of a quiver $Q$ with vertices $1,\dots,t$ (see Definition \ref{hro}).

$\Longrightarrow$. \
Let $\cal A$ and $\cal B$ be isometric; that is, there exist isometries
$\varphi_1:\mathcal U_1\to \mathcal V_1,\dots,
\varphi_t:\mathcal U_t\to \mathcal V_t$ such that $\varphi_v\mathcal A_{\alpha }=\mathcal B_{\alpha
}\varphi _u$ for each arrow $\alpha
:u\longrightarrow  v$. Let
\begin{equation*}\label{ryb}
\pi:\quad\xymatrix{
{v_1}
\ar@{->}@/_2pc/[rrr]^{\gamma_{\ell}}
\ar@{<-}[r]^{\gamma_1}&
v_2\ar@{<-}[r]^{\gamma_2\ \ } &
{\ \cdots\ }&{v_{\ell}}
\ar@{->}[l]_{\ \gamma_{{\ell}-1}\ }},\qquad {\ell}\ge 1
\end{equation*}
be a cycle in $\widetilde Q$ (thus, each $\gamma_i$ is either $\alpha_i$ or $\alpha_i^{\ast}$, where $\alpha_i$ is an arrow of $Q$). Then
\begin{align*}
\tr {\widetilde{\mathcal A}}(\pi)
&=
\tr \left(\widetilde{\mathcal A}_{\gamma_1}
\widetilde{\mathcal A}_{\gamma_2}\cdots \widetilde{\mathcal A}_{\gamma_{\ell}}\right)\\
&=
\tr \left(\varphi_{v_1}^{-1} \widetilde{\mathcal B}_{\gamma_1}\varphi_{v_2}\cdot\varphi_{v_2}^{-1}
\widetilde{\mathcal B}_{\gamma_2}\varphi_{v_3}\cdots\varphi_{v_{\ell}}^{-1} \widetilde{\mathcal B}_{\gamma_{\ell}}\varphi_{v_1}\right)
   \\
   &=
\tr \left(\varphi_{v_1}^{-1}\widetilde{\mathcal B}_{\gamma_1}
\widetilde{\mathcal B}_{\gamma_2}\cdots\widetilde{\mathcal B}_{\gamma_{\ell}}\varphi_{v_1}\right)\\
  &=
\tr \left(\widetilde{\mathcal B}_{\gamma_1}
\widetilde{\mathcal B}_{\gamma_2}\cdots\widetilde{\mathcal B}_{\gamma_{\ell}}\right)
=\tr \widetilde{\mathcal B}(\pi).
\end{align*}

$\Longleftarrow$. \
Let
\begin{equation}\label{kuye}
\text{$\tr \widetilde{\mathcal A}(\pi)=\tr \widetilde{\mathcal B}(\pi)$\quad for each oriented cycle $\pi$ in $\widetilde Q$.}
\end{equation}
Choosing orthonormal bases in the spaces of representations $\mathcal A$ and $\mathcal B$, we obtain two matrix representations $A$ and $B$ of $Q$. By Lemma \ref{njp}, it suffices to prove that the matrices $M_Q(A)$ and $M_Q(B)$ are $^{\pen}$unitarily similar. Due to Specht's criterion \cite{spe} for complex matrices under unitary similarity and its generalization by Jing \cite{jin} to complex matrices under complex orthogonal similarity (see Section \ref{intr}),  $M_Q(A)$ and $M_Q(B)$ are $^{\pen}$unitarily similar if and only if
\begin{equation*}\label{guw}
 \tr w(M_Q(A),M_Q(A)^{\pen}) = \tr w(M_Q(B),M_Q(B)^{\pen})
\end{equation*}
for every word $w(x,y)$.

Let us consider the matrix
\[W(A)=[W_{ij}(A)]_{i,j=1}^{r+2}:=w(M_Q(A),M_Q(A)^{\pen})\]
that is partitioned into $(r+2)^2$ blocks as \eqref{rkr}.
Since
\[
\tr W(A)=\tr W_{11}(A)+\tr W_{22}(A)+\dots
+\tr W_{r+2,r+2}(A),
\]
it suffices to prove that
\[
\tr W_{\ell
\ell}(A)=\tr W_{\ell\ell}(B)
\qquad \text{for all }\ell=1,\dots,r+2.
\]

Each $W_{\ell\ell}(A)$ is a linear combination of products of blocks of the form $0$, $\widehat D_{i}$, $\widehat{A}_{j}$, and $\widehat{A}_{k}^{\,\pen}$ in $W(A)$.
Thus, it suffices to prove that
\begin{multline*}\label{afd}
\tr v(\widehat D_1,\dots,\widehat D_{r+2};\, \widehat A_1, \dots,\widehat A_{\rho };\, \widehat A_1^{\,\pen}, \dots,\widehat A_{\rho }^{\,\pen})\\
=\tr v(\widehat D_1,\dots,\widehat D_{r+2};\, \widehat B_1, \dots,\widehat B_{\rho };\, \widehat B_1^{\,\pen}, \dots,\widehat B_{\rho }^{\,\pen})
\end{multline*}
for each word
\[
v(x_1,\dots,x_{r+2};\, y_1,\dots,y_{\rho};\,
z_1,\dots,z_{\rho}),\quad
\rho :={r(r+1)}/2.
\]

Let us consider the matrix
\[
V(A)=[V_{ij}(A)]_{i,j=1}^t:   =v(\widehat D_1,\dots,\widehat D_{r+2};\, \widehat A_1, \dots,\widehat A_{\rho };\, \widehat A_1^{\,\pen}, \dots,\widehat A_{\rho }^{\,\pen})
\]
that is partitioned into $t$ horizontal strips and $t$ vertical strips  of sizes $d_1,\dots,d_t$ (as the blocks of \eqref{lsk}). Since \[\tr V(A)=\tr V_{11}(A)+\tr V_{22}(A)+
\dots+\tr V_{tt}(A),\] it suffices to prove that
\begin{equation*}\label{xcv}
\tr V_{l
l}(A)=\tr V_{ll}(B)
\qquad \text{for all }l=1,\dots,t.
\end{equation*}

Each nonzero $V_{ll}(A)$ is a linear combination of products of the form
\begin{equation}\label{nlo}
A_{l i_2:\xi_1}^{(\varepsilon_1)}
A_{i_2 i_3:\xi_2}^{(\varepsilon_2)}
A_{i_3 i_4:\xi_3}^{(\varepsilon_3)}
\cdots
A_{i_k l:\xi_k}^{(\varepsilon_k)},
\end{equation}
in which $k\ge 1$, $i_2,\dots,i_k\in\{1,\dots,t\}$,
$\varepsilon_1,\dots,\varepsilon_k\in\{1,\pen\}$, and
\begin{equation*}\label{aamd}
A_{ij:\xi}^{(\varepsilon)}:=
  \begin{cases}
    A_{ij:\xi} & \hbox{if }\varepsilon =1 \\
    A_{ji:\xi}^{\pen} & \hbox{if }\varepsilon =\pen
  \end{cases}\qquad\text{(compare with \eqref{kmj}).}
\end{equation*}
We must prove that
\begin{equation}\label{lkj}
\begin{split}
  & \tr \left(A_{l i_2:\xi_1}^{(\varepsilon_1)}
A_{i_2 i_3:\xi_2}^{(\varepsilon_2)}
A_{i_3 i_4:\xi_3}^{(\varepsilon_3)}
\cdots
B_{i_k l:\xi_k}^{(\varepsilon_k)}\right) \\
    =& \tr\left(
B_{l i_2:\xi_1}^{(\varepsilon_1)}
B_{i_2 i_3:\xi_2}^{(\varepsilon_2)}
B_{i_3 i_4:\xi_3}^{(\varepsilon_3)}
\cdots
B_{i_k l:\xi_k}^{(\varepsilon_k)}\right).
\end{split}
\end{equation}

For each natural number $n$, let $[n]\in\{1,2,\dots,t\}$ be the vertex of $Q$ such that
$[n]\equiv n\pmod t$.
The matrices of the product \eqref{nlo} define the complex matrix representation
\begin{equation*}\label{rie}
\xymatrix@=70pt{
{d_{[l]}}
\ar@{->}@/_2pc/[rrr]
  ^{A_{i_k l:\xi_k}^{(\varepsilon_k)}}
\ar@{<-}[r]
  ^{A_{l i_2:\xi_1}^{(\varepsilon_1)}}&
d_{[i_2]}\ar@{<-}[r]
^{A_{i_2 i_3:\xi_2}^{(\varepsilon_2)}\ \ } &
{\ \cdots\ }&{d_{[i_k]}}
\ar@{->}[l]
  _{\ A_{i_{k-1} i_k:\xi_{k-1}}^{(\varepsilon_{k-1})}\ }}
\end{equation*}
of the oriented cycle
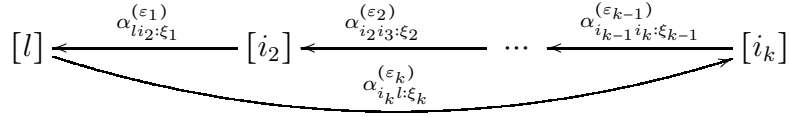
\[
\xymatrix@=70pt{
{[l]}
\ar@{->}@/_2pc/[rrr]
  ^{\alpha _{i_k l:\xi_k}^{(\varepsilon_k)}}
\ar@{<-}[r]
  ^{\alpha _{l i_2:\xi_1}^{(\varepsilon_1)}}&
[i_2]\ar@{<-}[r]
^{\alpha _{i_2 i_3:\xi_2}^{(\varepsilon_2)}\ \ } &
{\ \cdots\ }&{[i_k]}
\ar@{->}[l]
  _{\ \alpha _{i_{k-1} i_k:\xi_{k-1}}^{(\varepsilon_{k-1})}\ }}
\]
in $\widetilde Q$, in which $\alpha ^{(1)}:=\alpha$ and $\alpha ^{(\pen)}:=\alpha^{\star}$ for every arrow $\alpha $ of $Q$.

The equality \eqref{lkj}
holds by \eqref{kuye}, which proves the statement (a) in Theorem \ref{kke}.
\medskip

(b) The bound \eqref{wwk} holds since the matrix \eqref{rkr} is of size $n\times n$, in which $n=(r+2)(d_1+\dots+d_t)$.

\section*{Acknowledgements}

V. Futorny is
supported in part by  CNPq grant (301320/2013-6) and by
FAPESP grant (2014/09310-5). This work was done during the visit of V.V.~Sergeichuk to the University of S\~ao Paulo;
he is grateful to the university for hospitality and the FAPESP
for financial support (grant 2015/05864-9).


\begin{thebibliography}{aa}

\bibitem{arv0}
W.B. Arveson, Unitary invariants for compact operators, {Bull. Amer. Math. Soc.} 76 (1970) 88--91.



\bibitem{bel}
G.R. Belitskii, V.V. Sergeichuk, Complexity of matrix problems, Linear Algebra Appl. 361 (2003) 203--222.

\bibitem{far2}
D. Farenick, Arveson's criterion for unitary similarity, {Linear Algebra Appl.} 435 (2011) 769--777.

\bibitem{far}
D. Farenick, V. Futorny, T.G. Gerasimova, V.V.~Sergeichuk, N.~Shvai,
A criterion for unitary similarity of upper triangular matrices in general position, Linear Algebra Appl. 435 (2011) 1356--1369.

\bibitem{far1}
D. Farenick, T.G. Gerasimova, N. Shvai,
A complete unitary similarity invariant for unicellular matrices, Linear Algebra Appl. 435 (2011) 409--419.


\bibitem{gab}
P. Gabriel, Unzerlegbare Darstellungen I, Manuscripta Math. 6 (1972) 71--103.

\bibitem{ger}
T.G. Gerasimova, R.A. Horn, V.V. Sergeichuk, Simultaneous unitary equivalences, Linear Algebra Appl. 438 (2013) 3829--3835.


\bibitem{h-j} R.A. Horn, C.R. Johnson,
    Matrix Analysis, 2nd ed., Cambridge
    University Press, Cambridge, 2013.

\bibitem{jin}
N. Jing, Unitary and orthogonal equivalence of sets of matrices, Linear Algebra Appl. 481 (2015) 235--242.

\bibitem{kap}
I. Kaplansky, Linear Algebra and Geometry. A Second Course, Chelsea Publishing Company, NY, 1974.

\bibitem{laf}
T.J. Laffey, Simultaneous reduction of sets of matrices under similarity, Linear Algebra Appl. 84 (1986) 123--138.

\bibitem{lit}
D.E. Littlewood, On unitary equivalence, J. London Math. Soc. 28 (1953) 314--322.

\bibitem{pap}
C.J. Pappacena, An upper bound for the length of a finite-dimensional algebra, J. Algebra 197 (1997) 535--545.

\bibitem{pea}
C. Pearcy, A complete set of unitary invariants for operators generating finite $W^*$-algebras of type I, Pacific J. Math. 12 (1962) 1405--1414.

\bibitem{rin}
K.M. Ringel, Introduction to the Representation Theory of Quivers, A course in King Abdulaziz University,  Jeddah, KSA, 2012. Available from: \verb"https://www.math.uni-bielefeld.de/~sek/kau/"

\bibitem{ser_func}
V.V. Sergeichuk, Classification of linear operators in a finite-dimensional unitary space, Functional Anal. Appl. 18 (3) (1984) 224--230.

\bibitem{ser_unit}
V.V. Sergeichuk, Unitary and Euclidean representations of a quiver, Linear Algebra Appl. 278 (1998) 37--62.

\bibitem{sha}
H.  Shapiro,  A  survey  of  canonical forms  and  invariants  for  unitary  similarity,  Linear  Algebra Appl. 147 (1991) 101--167.

\bibitem{spe}
W. Specht, Zur Theorie der Matrizen II, Jahresber. Deutsch. Math.-Verein. 50 (1940) 19--23.

\bibitem{wie}
N.A. Wiegmann, Necessary and sufficient conditions for unitary similarity, J. Austral. Math. Soc. 2 (1961/1962) 122--126.
\end{thebibliography}
\end{document}